\documentclass[a4paper,12pt]{article}
\usepackage{a4wide}
\usepackage{amsmath}
\usepackage{amssymb}
\usepackage{amsthm}
\usepackage{latexsym}
\usepackage{graphicx}
\usepackage[english]{babel}
\usepackage{makeidx}
\usepackage{mathtools}

\newtheorem{obs} [subsection]{Remark}

\newtheorem{prop}[subsection]{Proposition}

\newtheorem{teor}[subsection]{Theorem}
\newtheorem{lema}[subsection]{Lemma}

\numberwithin{equation}{section}
\usepackage{titletoc}
\dottedcontents{section}[2.5em]{}{2em}{1pc}
\begin{document}
\selectlanguage{english}
\frenchspacing

\title{A note on the action of Hecke groups on subsets of quadratic fields}
\author{Mircea Cimpoea\c s}
\date{}

\maketitle

\begin{abstract}
We study the action of the groups $H(\lambda)$ generated by the linear fractional transformations 
$x:z\mapsto -\frac{1}{z} \text{ and }w:z\mapsto z+\lambda$, where $\lambda$ is a positive integer, 
on the subsets $\mathbb Q^*(\sqrt{n})=\{\frac{a+\sqrt n}{c}\;|\;a,b=\frac{a^2-n}{c},c\in\mathbb Z\}$,
where $n$ is a square-free integer. We prove that this action has a finite number of orbits if and
only if $\lambda=1$ or $\lambda=2$, and we give an upper bound for the number of orbits for $\lambda=2$.


\end{abstract}

\section*{Introduction}

E.\ Hecke introduced in \cite{hecke} the groups $H(\lambda)$ generated by two linear fractional transformations
$x:z\mapsto -\frac{1}{z} \text{ and }w:z\mapsto z+\lambda$, 
where $\lambda$ is positive real number. He further showed that $H(\lambda)$ is discrete if and only if 
$\lambda=\lambda_k=2\cos \left(\frac{\pi}{k}\right)$, $k\geq 3$ an integer, or $\lambda\geq 2$.
If $\lambda=\lambda_3=1$, then $H(1)=PSL(2,\mathbb Z)$ is the modular group
which consists in all the transformations $z\mapsto \frac{az+b}{cz+d}$ with $a,b,c,d\in\mathbb Z$ and $ac-bd=1$.

The actions of the modular group on many discrete and non-discrete structures play significant roles
in different branches of mathematics. Among these structures upon which the modular group acts are subsets 
of quadratic number fields. In \cite{mush}, Q. Mushtaq studied the action of $PSL(2,\mathbb Z)$ on the subset
$$\mathbb Q^*(\sqrt{n})=\left\{\frac{a+\sqrt n}{c}\;|\;a,b=\frac{a^2-n}{c},c\in\mathbb Z\right\},$$
where $n\geq 2$ is a square-free integer. Subsequent works by several authors considered properties emerging from this action; 
see for instance \cite{mush2} and \cite{malik}. Recently, M.\ Aslam and A.\ Deajim \cite{aslam} studied the action of $PSL(2,\mathbb Z)$
on $\mathbb Q^*(\sqrt{n})$, when $n$ is a negative squre-free integer, and give a formula for the number of orbits of
this action.

The aim of our paper is to study the action of the groups $H(\lambda)$, where $\lambda\geq 1$ is an integer, on the subsets 
$\mathbb Q^*(\sqrt n)$,
where $n$ is a square-free integer. We prove that the action of $H(\lambda)$ on 
$\mathbb Q^*(\sqrt n)$ has a finite number of orbits if and only if $\lambda\in \{1,2\}$;
see Theorem $1.4$. For $\lambda=2$ we give an upper bound for the number of orbits; see Theorem $1.5$.

\newpage
\section{Main results}

We fix $n$ a square-free integer and we consider the sets:
\begin{align*}
& \mathbb Q^*(\sqrt{n})=\left\{\alpha=\frac{a+\sqrt{n}}{c}\;:\;a,b=\frac{a^2-n}{c},c\in\mathbb Z \right\} \subset \mathbb Q(\sqrt{n}) \\
& A(n)=\left\{(a,b,c)\in\mathbb Z^3\;:\;bc=a^2-n \right\}.
\end{align*}
The map $\frac{a+\sqrt{n}}{c}\mapsto (a,b,c)$ gives an $1$-to-$1$ correspondence between $\mathbb Q^*(\sqrt{n})$ and $A(n)$.

\begin{prop}\label{p1}
For any integer $\lambda\geq 1$, the group $H(\lambda)$ acts on $\mathbb Q^*(\sqrt{n})$, hence on $A(n)$.
\end{prop}

\begin{proof}
Since $H(\lambda)$ is generated by $x:z\mapsto -\frac{1}{z}$ and $w_{\lambda}:z\mapsto z+\lambda$, it is enough to show that for
any $\alpha=\frac{a+\sqrt{n}}{c}\in \mathbb Q^*(\sqrt{n})$, it holds that $x(\alpha), w_{\lambda}(\alpha)\in \mathbb Q^*(\sqrt{n})$.
Indeed, we have
\begin{equation}\label{11}
x(\alpha)=-\frac{1}{\alpha}=-\frac{c}{a+\sqrt{n}}=-\frac{c(a-\sqrt{n})}{a^2-n}=\frac{-a+\sqrt{n}}{b}.
\end{equation}
On the other hand
\begin{equation}\label{12}
\frac{(-a)^2-n}{b} = \frac{a^2-n}{b}=c \in\mathbb Z.
\end{equation}
From \eqref{11} and \eqref{12} it follows that $x(\alpha)\in Q^*(\sqrt{n})$.

Also, we have that 
\begin{equation}\label{13}
w_{\lambda}(\alpha)=\alpha+\lambda = \frac{a+\lambda c + \sqrt{n}}{c}.
\end{equation}
On the other hand 
\begin{equation}\label{14}
\frac{(a+\lambda c)^2-n}{c} = \frac{a^2+2\lambda ac + \lambda^2c^2 - n}{c} = \frac{bc + 2\lambda ac +\lambda^2 c^2}{c} =
2\lambda a + b + \lambda^2 c \in \mathbb Z.
\end{equation}
From \eqref{13} and \eqref{14} it follows that $w_{\lambda}(\alpha)\in Q^*(\sqrt{n})$.
\end{proof}

From the proof of Proposition \ref{p1}, if follows that 
\begin{equation}\label{15}
x(a,b,c)=(-a,c,b)\text{ and }y(a,b,c)=(-a-\lambda c, c, 2\lambda a + b +\lambda^2 c).
\end{equation}

\begin{lema}
Let $\varepsilon:=\frac{5-\sqrt 7}{3}\approx 0.78$.
Let $(a,b,c)\in A(n)$ such that $c\geq \varepsilon |a|$ and $a^2-n>0$.
If $k$ is an integer with $|k|\geq 3$ and $(a',b',c'):=(xw_{k})(a,b,c)$ then $|a'|>|a|$ and $|c'|>\varepsilon |a'|$.
\end{lema}

\begin{proof}

Since $c>0$ and $bc=a^2-n>0$, we have $b>0$. From \eqref{15}, it follows that 
$$a'=-a-kc,\;b'=c,\;c'=b+2ka+k^2c.$$

Since $|kc|=|k|c \geq 3\varepsilon |a| > |a|$, it follows that 
$$|a'|\geq |k|c-|a| \geq (\varepsilon |k|-1)|a| > |a|.$$ 
In order to prove $|c'|>\varepsilon |a'|$, we consider two cases:
\begin{enumerate}
 \item[(i)] $ak > 0$. We have $|a'|=|a|+|k|c$. Hence:
            $$c'=b+2ka+k^2c>2|k||a|+k^2c=|k|(2a+|k|c)>|k|(|a|+|k|c)=|k||a'|>\varepsilon |a'|.$$
 \item[(ii)] $ak<0$. We have $|a'|= -|a|+|k|c$. On the other hand
            $$c'=b+2ka+k^2c>2ka+k^2c=-2|k||a|+|k|^2c.$$
            It follows that
            \begin{align*}                       
            & c'-\varepsilon |a'| > -2|k|a+|k|^2c + \varepsilon |a| +\varepsilon |k|c = (|k|^2-\varepsilon |k|)c+(-2|k|+\varepsilon)|a| \geq \\
            & (\varepsilon |k|^2-\varepsilon^2 |k| -2|k|+\varepsilon)|a| = \varepsilon(|k|^2 - \frac{2+\varepsilon}{\varepsilon}|k| + 1 )|a| =
            \varepsilon|a|(k-3)(k-\frac{1}{3}).             
            \end{align*}
            Since $|k|\geq 3$, we get $c'-\varepsilon |a'|>0$, thus $c'>\varepsilon |a'|$. (We used the identity $\frac{2+\varepsilon}{\varepsilon} = \frac{10}{3}$)
\end{enumerate}
\end{proof}

Given $\alpha:=(a,b,c)\in A(n)$, we denote $||\alpha||:=|a|$. Note that $||x(\alpha)||=||\alpha||=|a|$.

\begin{lema}
The set $B(n):=\{(a,b,c)\in A(n)\;|\;|a|\leq |b|,\;|a|\leq |c|\}$ is finite.
\end{lema}

\begin{proof}
We consider two cases:
\begin{enumerate}
 \item[(i)] $n>0$. Since $a^2-n=bc$, it follows that $|bc|=|a^2-n|\geq |a|^2$, hence $a^2\leq n$ and moreover $n-a^2\geq a^2$, 
            that is $|a|\leq \sqrt\frac{n}{2}$. It follows that 
            $$|b||c|=|a^2-n|\leq |\frac{n}{2}-n| = \frac{n}{2},$$
            hence $|b|\leq \frac{n}{2}$ and $|c|\leq \frac{n}{2}$.
 \item[(ii)] $n<0$. Since $bc=a^2-n>a^2$, $|b|>|a|$ and $|c|>|a|$ it follows that
            $$ a^2-n=|b||c|\geq (|a|+1)|a|=a^2+|a|,$$
            hence $|a|\leq -n$. Therefore, $|b||c|\leq n^2-n$, so $|b|,|c|\leq n^2-n$.
\end{enumerate}
\end{proof}

\begin{teor}
Let $\lambda\geq 1$ be an integer. 
The action of $H(\lambda)$ on $\mathbb Q^*(n)$ has a finite number of orbits if and only if $\lambda\in \{1,2\}$.
\end{teor}

\begin{proof}
Assume $\lambda\geq 3$. For any integer $s\geq 1$ with $(s-1)^2>n$, we let 
$$a_s:=s^2-s-n,\; b_s:=(s-1)^2-n,\; c_s:=s^2-n,\; \alpha_s:=(a_s,b_s,c_s).$$
Note that $a_s^2-n = b_sc_s$, hence $\alpha_s\in A(n)$. Since 
$$\lim_{s\to \infty} \frac{a_s}{c_s} =\lim_{s\to \infty} \frac{a_s}{b_s} = 1,$$
it follows that there exists an integer $s_0\geq 1$ with $(s_0-1)^2>n$, such that, for any $s\geq s_0$, we have $b_s>\varepsilon a_s$ 
and $c_s>\varepsilon a_s$, where $\varepsilon=\frac{5-\sqrt 7}{3}$.

Let $s>t\geq s_0$ be two integers such that $(t-1)^2>n$.

We claim that $H(\lambda)\alpha_s \neq H(\lambda)\alpha_{t}$.
In order to do that, we show that $\alpha_t\notin H(\lambda)\alpha_s$.
An element $\beta$ of the orbit $H(\lambda)\alpha_s$ is of the form $g\alpha_s$ with $g\in H(\lambda)$.
Since $H(\lambda)$ is generated by $x$ and $w_{\lambda}$, $g$ has one of the forms
\begin{align*}
& (i) g=xw_{k_1}xw_{k_2}x\cdots xw_{k_j}\\
& (ii) g=xw_{k_1}xw_{k_2}x\cdots xw_{k_j}x\\ 
& (iii) g=w_{k_1}xw_{k_2}x\cdots xw_{k_j}\\
& (iv) g=w_{k_1}xw_{k_2}x\cdots xw_{k_j}x,
\end{align*}
where $k_i$'s are nonzero integers, multiples of $\lambda$, and $j\geq 0$ is an integer.

Since $||x(\alpha)||=||\alpha||$ for any $\alpha\in A(n)$, it is enough to tackle the cases $(i)$ and $(ii)$. 
Note that $\alpha_s$ and $x\alpha_s$ satisfy the conditions of Lemma $1.2$. Hence, by applying Lemma $1.2$, it follows
that 
\begin{align*}
& ||(xw_{k_1}xw_{k_2}x\cdots xw_{k_j})(\alpha_s)||\geq a_s \\
& ||(xw_{k_1}xw_{k_2}x\cdots xw_{k_j}x)(\alpha_s)||\geq a_s. 
\end{align*}
Therefore, in all cases (i-iv), we have $||g(\alpha_s)||\geq ||\alpha_s||$. Hence $\alpha_t\notin H(\lambda)\alpha_s$, 
as required.

From the above argument, it follows that the orbits $H(\lambda)\alpha_{s_0}$, $H(\lambda)\alpha_{s_0+1}$, $H(\lambda)\alpha_{s_0+2}, \ldots$ are
distinct, hence we have an infinite number of orbits.

In order to complete the proof, as $H(2)$ is a subgroup of $H(1)$, it is enough to show that the action of $H(2)$ on $A(n)$ has a finite
number of orbits. Let $\alpha=(a,b,c)\in A(n)$. We claim that there exists $g\in H(2)$ such that $g(\alpha)\in B(n)$, where $B(n)$
was defined in the statement of Lemma $1.3$. We use the following algorithm:
\begin{enumerate}
 \item We let $g=1$.
 \item If $\alpha\in B(n)$, then we stop. 
 \item If $|a|>|c|$ then, according to the remainder theorem, there exists an integer $k\neq 0$ such that $-|c| <a+2kc \leq |c|$, that is $|a+2kc|\leq |c|$.
       We let $$\alpha':=w_{2k}(\alpha)=(a',b',c'),\text{ where } a'=a+2kc,\; b'=b+2ka+4k^2c \text{ and }c'=c.$$
       We have $|a'|=|a+2kc|\leq |c| < |a|$ and we replace $\alpha$ with $\alpha'$ and $g$ with $w_{2k}g$. We return to 2.
 \item If $|a|>|b|$ then, we let $\alpha'=x(\alpha)=(-a,c,b)$. We replace $\alpha$ with $\alpha'$ and $g$ with $xg$. We return to 3.
\end{enumerate}
Since, in the step 3., $||\alpha'||=|a'|<|a|=||\alpha||$, the above procedure eventually stop, and we obtain $g\in H(2)$ with $g(\alpha)\in B(n)$, as required.
Since, according to Lemma $1.3$, the set $B(n)$ is finite, it follows that the set of orbits of $H(2)$ is also finite.
\end{proof}


\begin{teor}
The number $o(n)$ of orbits of the action of $H(2)$ on $\mathbb Q^*(n)$ is:
$$ o(n) \leq \begin{cases} \sum_{j=0}^{\left\lceil \sqrt{\frac{n}{2}}-1 \right\rceil} d(n-j^2) , & n<0 \\
                           \sum_{j=1}^{n} d(j^2-n),& n>0\end{cases},$$
where $d(n)=$ the number of integer divisors of $n$.
\end{teor}

\begin{proof}
We consider the sets:
\begin{align*}
& B^+(n)=\{(a,b,c)\in B(n)\;|\; a>0\}\\
& B^0(n)=\{(a,b,c)\in B(n)\;|\; a=0\}\\
& B^-(n)=\{(a,b,c)\in B(n)\;|\; a<0\}
\end{align*}
Obviously, $|B^+(n)|=|B^-(n)|$ and $|B^0(n)|=|\{(0,b,c)\;|\;bc=-n\}|=d(n)$. 
According go the proof of Theorem $1.4$, for any $\alpha=(a,b,c)\in A(n)$, there exists a $g\in H(2)$ such that 
$g(\alpha)\in B(n)$. Note that, if $g(\alpha)\in B^-(n)$, then $xg(\alpha)\in B^+(n)$. It follows that
\begin{equation}\label{cuur}
 o(n)\leq |B^+(n)|+d(n).
\end{equation}
Let $(a,b,c)\in B^+(n)$.
We consider two cases:
\begin{enumerate}
 \item[(i)] $n<0$. As in the proof of Lemma $1.3$, we have $a\leq -n$. Since $bc=a^2-n$, it follows that $b$ is a divisor of $a^2-n$, hence $b$ can have $d(a^2-n)$ possible values.
            Also, $c$ is uniquely determined by $a$ and $b$. It follows that 
            \begin{equation}\label{cuur1}
            |B^+(n)| = \sum_{j=1}^{n} d(j^2-n).
            \end{equation}
 \item[(ii)] $n>0$. As in the proof of Lemma $1.3$, we have $a\leq \sqrt{\frac{n}{2}}$. In fact $a<\sqrt{\frac{n}{2}}$, otherwise $bc=0$ hence $b=0$ or $c=0$,
            a contradiction. Since $bc=a^2-n$, it follows that $b$ is a divisor of $n-a^2$, hence $b$ can have $d(n-a^2)$ possible values. Also, $c$ is uniquely determined by $a$ and $b$.
            It follows that 
\begin{equation}\label{cuur2}
|B^+(n)| = \sum_{j=1}^{\left\lceil \sqrt{\frac{n}{2}}-1 \right\rceil} d(n-j^2).
\end{equation}

\end{enumerate}
The conclusion follows from \eqref{cuur}, \eqref{cuur1} and \eqref{cuur2}. 
\end{proof}

\begin{obs}
 \emph{The inequality in Theorem $1.5$ is strict. Let $\alpha:=(1,-3,2)\in B^+(7)\subseteq A(7)$. We have $w_{-2}(\alpha)=(-3,1,2)$, $xw_{-2}(\alpha)=(3,2,1)$ and
$w_{-2}xw_{-2}(\alpha)=(1,-6,1)\in B^+(7)$.}

\emph{Let $\alpha:=(2,3,2)\in B^+(-2)\subset A(-2)$. We have $w_{-2}(\alpha)=(-2,3,2)$ and therefore $xw_{-2}(\alpha)=(2,2,3)\in B^+(-2)$.}
\end{obs}

{}

\vspace{2mm} \noindent {\footnotesize
\begin{minipage}[b]{15cm}
Mircea Cimpoea\c s, Simion Stoilow Institute of Mathematics, Research unit 5, P.O.Box 1-764,
Bucharest 014700, Romania and University Politehnica of Bucharest, Faculty of
Applied Sciences, Department of Mathematical Methods and Models, Bucharest, 060042, Romania.\\ 
 E-mail: mircea.cimpoeas@imar.ro
\end{minipage}}

\end{document}